\documentclass{amsart}
\input cyracc.def

\usepackage{amssymb}
\usepackage{amsthm}
\usepackage{color}

\theoremstyle{plain}
\newtheorem{thm}{Theorem}[section]
\newtheorem{lem}[thm]{Lemma}

\newtheorem{prop}[thm]{Proposition}
\newtheorem{conj}[thm]{Conjecture}

\theoremstyle{definition}
\newtheorem{dfn}[thm]{Definition}

\newtheorem{ex}[thm]{Example}

\def\dnfo{\;\raise.2em\hbox{$\mathrel|\kern-.9em\lower.4em\hbox
{$\smile$}$}}

\def\dnf#1{\lower.9em\hbox{$\buildrel\dnfo\over{ \scriptstyle  #1}$}}

\def\dfo{\;\raise.2em\hbox{$\mathrel|\kern-.9em\lower.4em\hbox{$\smile$}
\kern-.72em\lower.07em\hbox{\char'57}$}\;}
\def\df#1{\lower1em\hbox{$\buildrel\dfo\over{\scriptstyle #1}$}}

\newcommand{\la}{\langle}

\newcommand{\ra}{\rangle}
\newcommand{\sA}{\mathcal{A}}
\newcommand{\sB}{\mathcal{B}}
\newcommand{\sC}{\mathcal{C}}
\newcommand{\sD}{\mathcal{D}}
\newcommand{\sE}{\mathcal{E}}

\newcommand{\C}{\mathbb{C}}

\newcommand{\ov}{\overrightarrow}

\title{Generically Computable Equivalence Structures and Isomorphisms}
\thanks{This research was partially supported by the National Science Foundation SEALS grant NSF DMS-1362273.  The work was done partially while the latter two authors were visiting the Institute for Mathematical Sciences, National University of Singapore, in 2017. The visits were supported by the Institute.  Harizanov was partially supported by the Simons Foundation Collaboration Grant and by CCFF and Dean's Research Chair Award of the George Washington University.}

\author{Wesley Calvert}
\address{Department of Mathematics\\ Mail Code 4408\\
Southern Illinois University, Carbondale\\
1245 Lincoln Drive\\
Carbondale, Illinois 62901}
\email{wcalvert@siu.edu}
\author{Douglas Cenzer}
\email{cenzer@ufl.edu}
\author{Valentina S.\ Harizanov}
\email{harizanv@gwu.edu}
\date{\today}

\begin{document}

\begin{abstract}

We define notions of generically and coarsely computable relations and structures and functions between structures. 
We investigate the existence and uniqueness of equivalence structures in the context of these defintions.  

\end{abstract}

\maketitle

Many results in computable structure theory tend to depend sensitively on the construction of adversarial (and frequently \emph{ad hoc}) examples.  As a well-known example, a standard construction of a finitely presented group with unsolvable word problem \cite{Rotman} involves not just getting the right example of a group; the particular words within this group on which it is difficult to decide equality to the identity are very special words (and are even called by this term in some expositions).  In another well-known example from complexity theory, the simplex algorithm is known to have exponential complexity in the worst case, but empirically runs in much shorter time on practically all inputs.

It would be worthwhile to distinguish which results in computable structure theory depend on a ``special'' (and potentially extremely rare) input, and which are less sensitive.  To do this job in the context of word problems on groups, Kapovich, Myasnikov, Schupp, and Shpilrain proposed using notions of asymptotic density to state whether a partial recursive function could solve ``almost all'' instances of a problem \cite{KMSS}.

Jockusch and Schupp \cite{JS12} generalized this approach to the broader context of computability theory in the following way.
\begin{dfn}
  Let $S \subseteq \mathbb{N}$.
  \begin{enumerate}
  \item The density of $S$ up to $n$, denoted by $\rho_n(S)$, is given by \[\frac{\left|S \cap \{0, 1, 2, \dots, n\}\right|}{n+1}.\]
  \item The asymptotic density of $S$, denoted by $\rho(S)$, is given by $\lim\limits_{n \to \infty}\rho_n(S)$.
  \end{enumerate}
\end{dfn}

A set $A$ is said to be \emph{generically computable} if and only if there is a partial computable function $\phi$ such that $\phi$ agrees with $\chi_A$ throughout the domain of $\phi$, and such that the domain of $\phi$ has asymptotic density 1.  A set $A$ is said to be \emph{coarsely computable} if and only if there is a \emph{total} computable function $\phi$ that agrees with $\chi_A$ on a set of asymptotic density 1.

The study of generically and coarsely computable sets and some related notions has led to an interesting program of research in recent years; see \cite{JS17} for a partial survey.  The purpose of the present paper is to examine notions of generically and coarsely computable functions, relations, and structures and to present some results for equivalence structures and isomorphisms. 

Now it seems natural to say that a function $f: \omega^n \to \omega$ is generically computable if  there is a partial computable function $\phi$ such that $\phi = f$ on the domain of $\phi$ and such that the domain of $\phi$ has asymptotic density one. We will discuss in some detail below the notion of density for subsets of $\omega^n$. Given a structure $\sA$ with universe $\omega$, and with functions $\{f_i: i \in I\}$, each $f_i$ of arity $p_i$ and relations $\{R_j: j \in J\}$, each $R_i$ of arity $r_j$, we want to propose that $\sA$ is generically computable if there is a computably enumerable set $D$ of asymptotic density one, and partial computable functions $\{\phi_i: i \in I\}$ and $\psi_j$ such that each $\phi_i$ agrees with $f_i$ on $D^{p_i}$ and each $\psi_j$ agrees with $\chi_{R_j}$ on the set $D^{r_j}$.  We will present a number of variations on this theme. 

We will also consider versions of coarsely computable functions and structures.  Generalizing from the characteristic function of a set, we say that a function $f$ is coarsely computable if there is a computable function $\phi$ such that $f$ and $\phi$ agree on a set of asymptotic density one.  For a structure $\sA$, we want to say that $\sA$ is coarsely computable if there is a computable structure $\sC$ and a set $D$  of asymptotic density one such that the functions and relations of $\sA$ and $\sC$ agree on $D$. We introduce an intermediate notion of being \emph{strongly generically computable} which requires that the dense set $D$ is computably enumerable. When we examine equivalence structures, we will require that the computable structure $\sC$ also be an equivalence structure. 

Finally, we consider generically and coarsely computable isomorphisms. We will say that structures $\sA$ and $\sB$ are \emph{generically computably isomorphic} if there is an isomorphism $f: \sA \to \sB$ and a partial computable $\theta$ such that both the domain and range of $\theta$ have asymptotic density one, and $f(x) = \theta(x)$ whenever $\theta(x)$ is defined.  A bijection $f: \sA \to \sB$ is said to be a \emph{weakly coarsely computable isomorphism} if there is a total computable $\theta$ and a set $C$ of density one such that 
\begin{itemize}
\item[(i)] $C$ is the universe of a substructure $\sC$ of $\sA$;
\item[(ii)] $f(x) = \theta(x)$ for all $x \in C$;
\item[(iii)] $f[C]$ has asymptotic density one;
\item[(iv)] $\theta$ is a structural isomorphism from $\sC$ to its image.   
\end{itemize}
If the bijection $f$ is itself a structural isomorphism, then $f$ is said to be a \emph{coarsely computable isomorphism}.

These notions prove quite interesting for equivalence structures.  Effective notions of equivalence relations and isomorphisms have been well-studied in recent years. See for example \cite{CHR11,CCHM06,KT09,Marshall,DMN}.  Equivalence structures may be characterized by the number of equivalence classes of each cardinality. The character $\chi(\sE)$ gives the number of classes of size $k$ for each finite $k$. We will examine in some detail the notions of a generically computable and a coarsely computable structure for equivalence relations. 
A key example from computable model theory is the $(1,2)$-structure, consiting of infinitely many classes of size one and infinitely many classes of size 2. The elements of $\sA$ belonging to classes of size 2 form a c.e. set $\sA(2)$ and the elements belonging to classes of size 1 form a co-c.e. set $\sA(1)$.  There are $(1,2)$-structures $\sA$ and $\sB$ such that $\sA(1)$ is computable but $\sB(1)$ is not computable and therefore these structures are not computably isomorphic.  We will say that an equivalence structure has \emph{generic character} $\{k\}$ if $\sA(k)$ has asymptotic density one. We will show that if $\sA$ and $\sB$ are computable $(1,2)$-structures each having generic character $\{2\}$, then they are generically computably isomorphic.  However there are computable $(1.2)$-structures each having generic character $\{1\}$ which are not generically computably isomorphic, although they will be coarsely computably isomorphic. 
We also consider $(1,2)$-structures in which the asymptotic density of $\sA(k)$ is some computable number between 0 and 1. 

The outline of this paper is as follows.  Section 1 contains the needed definitions and some key lemmas. We show that a set $A$ has asymptotic density $\delta$ if and only if the set $A \times A$ has density $\delta^2$ in $\omega \times \omega$. We show that there is a computable dense set $C \subset \omega \times \omega$ such that for any infinite computably enumerable set $A$, the product $A \times A$ is not a subset of $C$. We extend the lemma from \cite{CCHM06} to show that any computably enumerable equivalence relation on a computably enumerable set, with no infinite equivalence classes and with unbounded character, possesses an $s_1$-function (a technical auxilliary that is frequently useful in this area, which we will define). 

 Section 2 presents definitions and results for generically computable structures, in parictular for equivalence structures. We will say that a binary relation $R$ is \emph{generically computable} if there is a partial computable function $\phi: \omega \times \omega \to 2$ such that $\phi = \chi_R$ on the domain of $\phi$ and there is a computably enumerable set $A$ of asymptotic density one such that $A \times A \subseteq Dom(\phi)$. We present the surprising result that \emph{every} equivalence relation has a generically computable copy.  We will say that a set $A$ is $R$-faithful for an equivalence relation $R$ if whenever $a \in A$ and either $R(a,b)$ or $R(b,a)$, then $b \in A$; then we say that $R$ is \emph{faithfully generically computable} if the computably enumerable set $A$ above is $R$-faithful. We characterize the equivalence structures which have faithfully generically computable copies in several ways.  

Here is an abbreviated version of the result:

\bigskip

{\bf Theorem}  Let $\sE = (\omega,E)$ be an equivalence structure. Then the following are equivalent:
\begin{enumerate}
\item [(a)] $\sE$ has a faithfully strongly generically computable copy;
\item [(b)] $\sE$ has a faithfully  generically computable copy;
\item [(c)] $\sE$ has an infinite faithful substructure with a computable copy;
\item  [(d)] Either (i) $\sE$ has an infinite equivalence class, or (ii) there is a finite $k$ such that $\sE$ has infinitely many classes of size $k$, or (iii) $\chi(\sE)$ has an infinite $\Sigma^0_2$ subset with an $s_1$-function.  
\end{enumerate}

\medskip

In Section 3, we discuss coarsely computable structures.  We will say that a binary relation $E$ is \emph{coarsely computable} if there is a computable relation $R$ and a set $A$ of asymptotic density one such that $R$ and $E$ agree on $A$. If the set $A$ is $R$-faithful and also $E$-faithful, then $E$ is \emph{faithfully coarsely computable}; for an equivalence relation $E$, we require that $R$ also be an equivalence relation. 
Then every (faithfully) generically computable equivalence structure is also (faithfully) coarsely computable.  We construct a family of examples to show that not every faithfully coarsely computable structure has a faithfully generically computable copy.  We also show that not every equivalence structure has a faithfully coarsely computable copy.

In Section 4, we study generically computable and coarsely computable isomorphisms.  In addition to the results mentioned above, we also prove the following: 

\bigskip

{\bf Theorem} Suppose that $\sA$ and $\sB$ are computable $(1,2)$-equivalence structures with universe $\omega$ such that the asymptotic density of $A(1)$ and $B(1)$ both equal the same computable real $q$.
Then $\sA$ and $\sB$ are weakly coarsely computably isomorphic. 

\section{Background and Preliminaries}

In \cite{JS12}, Jockusch and Schupp give the following
definitions.

\begin{dfn} Let $S \subseteq \omega$.
\begin{enumerate}
\item We say that $S$ is \emph{generically computable} if there  is a partial computable function $\Phi: \omega \to 2$ such that
  $\Phi = \chi_S$ on the domain of $\Phi$, and such that the domain of
  $\Phi$ has asymptotic density 1.
\item We say that $S$ is \emph{coarsely computable} if there is a
  computable set $T$ such that $S \triangle T$ has asymptotic density 0.
\end{enumerate}
\end{dfn}

It was shown in \cite{JS12} that there is a coarsely computable
  computably enumerable\ set which is not generically computable, and a generically
  computable computably enumerable\ set which is not coarsely computable.

The following observations will be useful. Let us say that $A$ has \emph{upper density 1} if $\limsup\limits_n \frac{\left|(A \cap n)\right|}n = 1$. 
Equivalently, there is a sequence $n_0 < n_1 < \dotsb$ such that  $\lim\limits_i \frac{\left|A \cap n_i\right|}{n_i} = 1$. 

\begin{lem} \label{lem1} If $A$ is a  computably enumerable set with upper density one, then $A$ has a  computable subset with upper density one. 
\end{lem}

\begin{proof} Suppose that $A$ is a computably enumerable set with upper density 1. Define  computable sequences $n_0,n_1, n_2, \dots$  and $s_0,s_1, s_2, \dots$ as follows. Let $n_0 = s_0= 0$. 
Let $s_1$ be the least $s$ such that, for some $n < s$,  we have $\left|n \cap A_s\right| \geq \frac12 n$, and let $n_1$ be the least such $n$.  Given $n_k$ and $s_k$, 
let $s_{k+1}$ be the least $s$ such that, for some $n$ with $n_k < n < s$, we have $\left|(n - n_k) \cap A_s\right| \geq \frac {2^{k+1} - 1}{2^{k+1}} (n - n_k)$, and let $n_{k+1}$ be the least such $n$. The computable dense set $B \subseteq A$ is defined so that, for each $i$, if $n_k \leq i < n_{k+1}$, then $i \in B \iff i \in A_{n_{k+1}}$. 
It follows from the construction that, for each $k$, the density of $B$ in $\{i: i>n_k\}$ is at least $\frac{2^k - 1}{2^k}$, so that $B$ has upper density 1. 
\end{proof}

In order to study binary relations and the corresponding structures, we need to look at notions such as generic computability for such relations.  

\begin{lem} \label{lem2} Let $A \subset \omega$.  Then $A$ has asymptotic density $\delta$ if and only if $A \times A$ has asymptotic density $\delta^2$ in $\omega \times \omega$. In particular,
$A$ is asymptotically dense in $\omega$ iff $A \times A$ is asymptotically dense in $\omega \times \omega$.
\end{lem}  

\begin{proof} Let $\delta_A(n) = \frac{\left|A \cap n\right|}{n}$ and let $\delta(n) =   \frac{\left|(A \times A) \cap (n\times n)\right|}{n^2}$ .
Since  $(A \times A) \cap (n\times n) = (A \cap n) \times (A \cap n)$, it follows that $\left|(A \times A) \cap n\times n\right| = \left|A \cap n\right|^2$
and hence $\delta(n) =  \delta_A(n)^2$.  If $\lim\limits_n \delta_A(n) = \delta$ exists, then $\lim\limits_n \delta(n) = \lim\limits_n \delta_n(A)^2 = \delta^2$. 
Conversely, if $\lim\limits_n \delta(n) = L = \delta^2$ exists, then $\lim\limits_n \delta_A(n) = \lim\limits_n \sqrt{\delta_n(A)} = \sqrt{L} = \delta$.
\end{proof}

A similar result holds for the density of $A^r$ in $\omega ^r$. On the other hand, we have the following. 

\begin{thm} \label{thm1} There is a computable dense $C \subset \omega \times \omega$ such that for any infinite computably enumerable set $A \subset \omega$, the product $A \times A$ is not a subset of $C$. 
\end{thm}

\begin{proof} Define $C$ as follows. For any pair $(a,b)$ with $max\{a,b\} = m$, proceed as follows. For each $e < m$, look for the first element $n > 2^e$ which has come in by stage $m$; call this $n_e$ if it exists. Then put $(a,b) \in C$, unless either $a = n_e$ or $b = n_e$ for some $e < m$. If $W_e$ is infinite, then it contains some element $n_e > 2^e$ which is the first to come into $W_e$ at some stage $s_e$, and then there will be another $n \in W_e$ which is greater than $s_e$ but $(n_e,n)$ will not be in $C$. The set $C$ is dense since there are at most $i$ elements less than $2^i$ of the form $n_e$ for any $e<i$ so that $C$ contains at least $(2^i - i)^2$ elements out of the $2^{2i}$ possible pairs up to $2^i$.
\end{proof}

We will be studying equivalence relations, so the following definitions are needed. An equivalence structure $\sA = (A,R)$ is simply a set with an equivalence relation $R$ on $A$. 

\begin{dfn} For any equivalence structure $\sA = (A,R)$, the \emph{character} $\chi(\sA)$ of $\sA$ is 
$\{(k,n): \sA \ \text{has at least $n$ equivalence classes of size $k$}\}.$
\end{dfn}

We will sometimes just refer to the character of $R$ when the set $A$ is implicit. 

\begin{dfn} The function $f: \omega^2 \to \omega$ is said to be an \emph{$s_1$-function} if the following hold: 
\begin{enumerate}
\item For every $i$ and $s$, $f(i,s) \leq f(i,s+1)$.
\item For every $i$, the limit $m_i = \lim\limits_{s\to \infty} f(i,s)$ exists. 
\item For every $i$, $m_i < m_{i+1}$. 
\end{enumerate}
\end{dfn}

The character $K$ is said to \emph{possess} the $s_1$-function $f$ if it has an equivalence class of size $m_i$ for each $i$. 
Here are some useful results about the characters of equivalence relations.

The first is a slight improvement of Lemma 2.1(c) of \cite{CHR11}.

\begin{lem} \label{lem3} For any computably enumerable equivalence relation $R$ on a computably enumerable set $A$, the character $\chi(R)$ is a $\Sigma^0_2$ set. 
\end{lem}

\begin{proof} The Lemma from \cite{CHR11} applies to a structure with universe $\omega$.  If $R$ is only defined on the computably enumerable set $A$, just let $S(x,y) \iff \left(R(x,y) \ \lor\ x=y\right)$.  This adds some classes of size 1 to the character, so that $\chi(S)$ is $\Sigma^0_2$ if and only if $\chi(S)$ is $\Sigma^0_2$. 
\end{proof}

The next lemma is part of Lemma 2.8 of \cite{CCHM06}.

\begin{lem} \label{lem4}
For any $\Sigma^0_2$ character $K$ which possesses a computable $s_1$-function, there is a computable equivalence structure $\sE$ with character $K$ and no infinite equivalence classes. 
\end{lem}

The next result is a variation of Lemma 2.6 of \cite{CCHM06}.  It follows from the previous Lemma \ref{lem4} that it also holds for structures restricted to a computably enumerable universe. 
 
\begin{lem} \label{lem5} Let $\mathcal{A} = (\omega,E)$ be a computably enumerable equivalence structure with
no infinite equivalence classes and an unbounded character. Then
 there is a computable $s_{1}$-function $f$ such that $\mathcal{A}$ contains an
equivalence class of size $m_{i}$ for all $i$, where $m_i = lim_{s}f(i,s)$.
\end{lem}

\begin{proof} Let $E^p$ be the $p^{th}$ stage in the enumeration of $E$, so that $E = \cup_p E^p$. 
We will define a uniformly computable family $a_{i}^{s}$ for $i\leq s$
in such a way that $a_{i}=lim_{s}a_{i}^{s}$ exists. We will also define a
computable sequence $p_{s}$, and let 
\[
f(i,s)=\left|\{a\leq p_{s}:aE^{p_s}a_{i}^{s}\}\right|.
\]
Hence, we will have 
\[
m_{i}=\lim\limits_{s}\;\left(\left|\{a\leq p_s :aE^{p_s} a_{i}\}\right|=\left|[a_{i}]\right|\right).
\]

At \emph{stage} $0$, we have $p_{0}=0$ and $a_{0}^{0}=0$, so $f(0,0)=1$.
In fact, $a_0^s$ will equal 0 for all $s$.

After stage $s$, we have $p_{s}$ and $a_{0}^{s},\dots ,a_{s}^{s}$ with $f(i,s)$ as above such that 
\[
f(0,s)<f(1,s)<\dotsb <f(s,s).
\]

At \emph{stage} $s+1$, we define the least $p > p_s$ and the lexicographically least sequence $b_0,\dots,b_{s+1}$ such that for
all $i \leq s$,
\[
f(i,s) \leq \left|\{a \leq p: a E^p b_i\}\right| < \left|\{a \leq p: a E^p b_{i+1}\}\right|,
\]
as follows.  Let $b_0 = a_0 = 0$.  Furthermore, $b_i = a_i^{s+1}$ whenever there do not exist a pair
$a,j$ with $j \leq i$, $aE^p a^s_j$ and $p_s < a \leq p$.
Then we let $a_i^{s+1} = b_i$ for each $i$ and let $p_{s+1} = p$. 

To see that such $p$ exists, let $m$ be the largest such
that $[a_{j}^{s}]=\{a \leq p_s: a E^{p_s} a_{j}^{s}\}$ for all $j\leq m$, and 
let $b_{i}=a_{i}^{s}$ for all $i \leq m$. 
Then use the fact that $\chi (\mathcal{A})$ is unbounded to find $b_{m+1},\dots ,b_{s+1}$ with 
\[
\left|[a_{m}^{s}]\right| < \left|[b_{m+1}]\right|<\left|[b_{m+2}]\right|<\dotsb < \left|[b_{s+1}]\right|,
\]
and take $p$ large enough so that $[b_{i}]=\{a \leq p: aE^p b_{i}\}$.

Finally, we verify that $a_i = lim_s a_i^s$ exists for each $i$.  Since there is no $j <0$, it follows from the construction that $a_0^s = 0$ for all $s$. Given $t$ such that $a_i = lim_s a_i^s$ has converged by stage $t$ for all $i \leq k$, let $r \geq t$ be large enough so that 
\[[a_i] = \{a < p_r: a E^{p_r} a_i\}\] for all $i \leq k$. (This uses the fact that there are no infinite classes.)  It follows from the construction that $a_{i+1}^s = a_{i+1}^r$ for all $s > r$.  
\end{proof}

\begin{prop} If $\sE$ is a computably enumerable equivalence structure with no infinite equivalence classes, then $\sE$ is isomorphic to a computable structure.
\end{prop}

\begin{proof} By Lemma \ref{lem3}, $\sE$ has a $\Sigma^0_2$ character, and by Lemma \ref{lem5}, this character possesses a computable $s_1$-function.  Then by Lemma \ref{lem4}, there is a computable structure with the same character and no infinite equivalence classes, and hence isomorphic to $\sE$. 
\end{proof} 

This last result also holds for a computably enumerable structure $\sE = (A,E)$ where $A$ is a computably enumerable set.

\section{Generically Computable Equivalence Structures}
In this section, we define the notion of a generically computable relation and, in particular, of a generically computable equivalence relation. In consideration of Lemma \ref{lem2} and Theorem \ref{thm1}, we look for a dense set $A$  in the domain so that the relation is computable on $A \times A$ rather than for a dense set in $\omega \times \omega$ where the relation is computable.  


\begin{dfn} \label{def:faithful} If $R$ is a relation on $\omega$ and $A$ is a subset of $\omega$, we say that $A$ is \emph{$R$-faithful} if 
whenever $a \in A$ and $R(\ov{b})$ for some tuple including $a$, then every element of $\ov{b}$ is in $A$.  
\end{dfn}

Thus if $R$ is an equivalence relation and $A$ is $R$-faithful, then for any $R$-equivalence class $C$, either $C \subseteq A$ or $C \cap A = \emptyset$. The following notions of generic computability seem to be the most appropriate in terms of the results obtained.  The first notion applies for any relation but the other two are specific to equivalence relations. 

\begin{dfn} \label{def:genrel} Let $R$ be a relation on $\omega$.  
\begin{enumerate}
\item We say that a relation $R \subseteq \omega^r$ is a \emph{generically computable relation} if there is a partial computable function $\Phi: \omega^r \to 2$ such that  $\Phi = \chi_R$ on the domain of $\Phi$, and there is a computably enumerable set $A$ of asymptotic density 1 such that $A^r$ is a subset of the domain of  $\Phi$.  We say that $R$ is \emph{faithfully}  generically computable if the set $A$ is $R$-faithful.
\item We say that $R$ is a \emph{strongly generically computable equivalence relation} if there is a computable equivalence relation $S$ on $\omega$ and a computably enumerable set $A$ of asymptotic density 1 such that $R$ agrees with $S$ on $A \times A$.  We say that $R$ is \emph{faithfully} strongly generically computable if $A$ is $R$-faithful and also $S$-faithful. 
\item We say that $R$ is a \emph{generically computably enumerable equivalence relation} if there is a computably enumerable equivalence relation $S$ on a computably enumerable set 
$B$ and a computably enumerable set $A \subseteq B$ of asymptotic density 1 such that $R$ agrees with $S$ on $A \times A$.  We say that $R$ is \emph{faithfully} generically computably enumerable if $A$ is $R$-faithful and also $S$-faithful. 
\end{enumerate}
\end{dfn}
 
It is easy to see that strongly generically computable implies generically computable, which implies generically computably enumerable  (For the latter, take the reflexive, transitive, symmetric closure of $\{(x,y): x,y \in A\ \&\ \Phi(x,y) = 1\}$.)

\begin{thm} \label{thm2}
 If an equivalence structure $\sE = (\omega,E)$ is generically
 computable (respectively, generically computably enumerable), then there is some infinite computable
 $Y \subseteq \omega$ such that the restriction of $E$ to $Y \times Y$ is computable (respectively, computably enumerable). 
\end{thm}

\begin{proof} Let $\Phi$ be the partial computable function and let $A$ be an asymptotically dense computably enumerable set,  given by the definition above. 
Then, by Lemma \ref{lem1}, $A$ has a computable subset $Y$ with upper density 1 (and thus infinite) with $Y \times Y \subseteq Dom(\Phi)$. Then $\chi_E = \Phi$ on the computable set $Y$.
\end{proof}

Note that the set $Y$ from the proof of Theorem \ref{thm2}  may not be a \emph{faithful} substructure of $\sE$.

\begin{ex} \label{ex1}  Let $K = \{(1,k): k \in C\}$ where $C$ has no infinite $\Sigma^0_2$ subset. Also take
an immune set $B$. Then define $\sE$ so that $B$ is one infinite class, and $\omega \setminus B$ has character $K$.
Then, while $\sE$ itself need not be computable, $\sE$ has a generically computable copy, where the infinite class is a dense computable set. 
Now let  $Y$ be an infinite computable subset of $\omega$. 
Since $B$ is immune, $Y \setminus B$ is infinite, so that $Y$ has infinitely many elements with finite equivalence classes.  If $(Y,E)$ has a computable copy, then this copy has a $\Sigma^0_2$ character which is a subset of $C$. Thus at least $(Y,E)$ is not a faithful substructure. 
\end{ex}

Our first result was unexpected.   

\begin{prop} \label{prop1} Every equivalence structure $\sE = (\omega,E)$ has a strongly generically computable copy.
\end{prop}

\begin{proof} The proof is by cases.  Suppose first that $\sE$ has an infinite equivalence class and let $B$ be such a class.  Let $A$ be a computable dense set. Then we can define a generically computable copy ${\sE}^* = (\omega,E^*)$ of $\sE$ so that $A$ is an infinite equivalence class and $(\omega \setminus B, E)$ is isomorphic to $(\omega \setminus B,E^*)$. The substructure $(B,E)$ is faithful, so that in the generically computable copy elements of $A$ are never equivalent to elements of $\omega \setminus A$. To see that $\sE^*$ is strongly generically computable, let $S = (A \times A) \cup (\omega \setminus A) \times (\omega \setminus A)$ be the equivalence relation with two infinite classes, $A$ and $\omega \setminus A$. Then $S$ is computable and it agrees with $E^*$ on the computable dense set $A$. In this case, the copy is faithfully strongly generically computable. 

Next, suppose that $\sE$ has no infinite equivalence class and therefore $\chi(\sE)$ is infinite.  There are two subcases.  First suppose that there is a finite $k$ such that $\sE$ has infinitely many classes of size $k$, and let $B \subset \omega$ be such that $(B,E)$ has character $K = \{(k,n): n>0\}$ .  Then there is a computable structure $(A,R)$ with character $K$, and we may take $A$ to have asymptotic density one and be coinfinite.  Then we may build a relation $E^*$ on $\omega \setminus A$ so that $(\omega \setminus A,E^*)$ is isomorphic to $(\omega \setminus B, E)$.  Once again, the substructure $(B,E)$ is faithful, so that in the generically computable copy elements of $A$ are never equivalent to elements of $\omega \setminus A$. Now the computable structure $(\omega,S)$ with equivalence classes $A_n = \{nk,nk+1,\dots,nk+k-1\}$ for each $n$ agrees with $E^*$ on the dense set $A$, so that $\sE^*$ is strongly faithfully generically computable. 

If there are no infinite classes, and no fixed $k$ with infinitely many classes of size $k$, then the character must be unbounded, that is, there must be infinitely many different $k$ such that $\sE$ has an equivalence class of size $k$. Choose one such class $B_k$ for each $k$, and let $B \subseteq \omega$ consist of exactly one element from each class $B_k$. Then the substructure $(B,E)$ consists of infinitely many classes of size one. We may assume, without loss of generality, that $B$ is coinfinite. Now let $A \subset \omega$ be a computable, co-infinite set of asymptotic density one, and let $f$ be a permutation of $\omega$ mapping $A$ onto $B$, and thus mapping $\omega \setminus A$ onto $\omega \setminus B$. Then we may define a copy of $\sE$ by letting $xRy \iff f(x) E f(y)$.  Then $R$ is computable on the computable, dense set $A$, since for $x,y \in A$, we have $xRy \iff x = y$.  In this case, $(A,E)$ is not necessarily a faithful substructure of $\sE$. To see that $\sE^*$ is strongly genericallly computable, let $S = \{(x,x): x \in \omega\}$, so that all classes of $S$ have size one. Then the computable structure $(\omega,S)$ agrees with $\sE^*$ on the computable dense set $A$. 
\end{proof}
 
So it seems that the notion of a generically computable equivalence structure is a bit too broad. Next, we consider faithful generically computable structures.

\begin{thm} \label{thm4} Let $\sE = (\omega,E)$ be an equivalence structure. Then the following are equivalent:
\begin{enumerate}
\item [(a)] $\sE$ has a faithfully strongly generically computable copy;
\item [(b)] $\sE$ has a faithfully  generically computable copy;
\item [(c)] $\sE$ has a faithfully generically computably enumerable copy;
\item [(d)] $\sE$ has an infinite faithful substructure with a computable copy;
\item [(e)] $\sE$ has an infinite faithful substructure with a computably enumerable copy;
\item  [(f)] Either (i) $\sE$ has an infinite equivalence class, or (ii) there is a finite $k$ such that $\sE$ has infinitely many classes of size $k$, or (iii) $\chi(\sE)$ has an infinite $\Sigma^0_2$ subset with an $s_1$-function.  
\end{enumerate}
\end{thm}

\begin{proof} We will show that $(a) \implies (b) \implies (c) \implies (f) \implies (a)$ and that $(d) \implies (e) \implies (f) \implies (d)$.  
As remarked above, it is easty to see that (a) implies (b), and that (b) implies (c).    

To show that (c) implies (f), we may assume without loss of generality that $\sE = (\omega,R)$ is faithfully generically computably enumerable.  Let $S$ be a computably enumerable equivalence relation on a computably enumerable set $B$, where there is a faithful computably enumerable set $A$ of density one where $R$ agrees with $S$. If $\sA = (A,S)$ has an infinite equivalence class, then certainly $\sE$ has an infinite equivalence class. If not, then $\chi(\sA)$ is an infinite 
$\Sigma^0_2$ set. Since $\sA$ is faithful, it follows that $\chi(\sA) \subset \chi(\sE)$. If there are arbitrarily large finite classes, then $\chi(\sA)$ must have an $s_1$-function by Lemma \ref{lem5}.  If not, then there must be a finite $k$ and infinitely many classes of size 
$k$ in $\sA$. Again, since $A$ is faithful, $\sE$ also has infinitely many classes of size $k$. 

To see that (f) implies (a), first observe that in cases (i) and (ii), the proof of Proposition \ref{prop1} yields a faithfully strongly generically computable copy.  So we may suppose that $\chi(\sE)$ has no infinite classes, has no finite $k$ with infinitely many classes of size $k$, and has an infinite $\Sigma^0_2$ subset $K$ which possesses an $s_1$-function. Now, by Lemma \ref{lem4}, there is a computable structure with character $K$.  If $\chi(\sE) \setminus K$ is finite, then, in fact, $\sE$ has a computable copy.  Otherwise, we may take a computable structure $\sA = (A,R)$ with character $K$, where the computable set $A$ has asymptotic density one and is co-infinite. Let $B$ be a faithful substructure of $\sE$ with character $K$ and let $f$ be an isomorphism of the equivalence structures $(A,R)$ and $(B,E)$, so that $f$ is a permutation of $\omega$ mapping $A$ onto $B$, and thus mapping $\omega \setminus A$ onto $\omega \setminus B$. Then we may extend $\sA$ to a generically computable structure $(\omega,S)$ by letting $x S y \iff f(x) E f(y)$.  For $x,y \in A$, we have $x S y \iff f(x) E f(y) \iff x R y$. Thus $(A,R)$ is a faithful computable substructure of $(\omega,S)$, as desired. 

Certainly, (d) implies (e).  To see that (e) implies (f), let $\sA = (A,E)$ be an infinite faithful substructure of $\sE$ and let $\sB = (\omega,S)$ be a computably enumerable copy of $\sA$. We will assume that $\sE$ (and hence $\sB$ also) has no infinite equivalence class and has no finite $k$ with infinitely many classes of size $k$.  Then, by Lemma \ref{lem5}, it follows that $\sB$ (and hence $\sA$ also) has a $\Sigma^0_2$ character $K$ which possesses a computable $s_1$-function.  Since $A$ is faithful, it follows that $K$ is a subset of $\chi(\sE)$. 

To see that (f) implies (d), first note that this is trivial in cases (i) or (ii).  In case (iii), just let $\sB$ be a faithful substructure of $\sE$ with a $\Sigma^0_2$ character $K$ which possesses a computable $s_1$ function. Then $\sB$ has a computable copy by Lemma \ref{lem4}. 
\end{proof}

\section{Coarsely Computable Equivalence Structures}

In this section, we examine the notion of coarsely computable equivalence structures. 

\begin{dfn} Let $\sE = (\omega,E)$ be an equivalence structure on $\omega$.  
We say that $\sE$ is a \emph{coarsely computable equivalence structure} if there is a computable equivalence relation $R$ and a set $A$ of density one such that for $a,b \in A$, $aRb \iff aEb$. If the set $A$ is both $E$-faithful and $R$-faithful, then we say that $\sE$ is \emph{faithfully coarsely computable}.  

\end{dfn}

It is clear from the definitions that a strongly generically computable structure is also coarsely computable. Hence every equivalence structure $\sE = (\omega,E)$ has a coarsely computable copy.  Also, any structure meeting condition (f) of Theorem \ref{thm4} has a faithfully coarsely computable copy.  

We will show that not every faithfully coarsely computable structure has a faithfully generically computable copy, and that not every equivalence structure has a faithfully coarsely computable copy. 

Let $(\omega,E)$ be the canonical structure with one class of every finite size $k$. 
The equivalence classes of $(\omega,E)$ are $\left\{\{0\}, \{1,2\}, \{3,4,5\}, \dots\right\}$.  The first $k$ classes have $1 + 2 + \dotsb + k = k(k+1)/2$ elements.  Let $K$ be any set and let $A_K$ be the classes of size $k$ for $k \in K$, under $E$. 

\begin{lem} If $K$ is a dense set, then $A_K$ is also a dense set.  
\end{lem}

\begin{proof}  Suppose that the complement of $K$ contains $m$ out of the first $n$ positive numbers.
Then the classes of size $k$ with $k \in K \cap \{1,2,\dots,n\}$ contain at most $n + (n-1) + \dots +(n -m+1) = m (2n-m+1)/2$ elements out of a total of $1 + 2 + \dots +n = n (n+1)/2$.  Then the ratio is $\frac  mn \cdot \frac{2n-m+1}{n+1} \leq 2m/n$.  Thus, if 
$\omega \setminus K$ has density zero, then $A_K$ will have density 1. 
\end{proof}

\begin{lem} For any dense co-infinite set $K$, there is a faithful coarsely computable structure with character $\{(k,i):  k \in K, i \leq 2\}$.
\end{lem}

\begin{proof}  Let $(\omega,E)$ be the canonical computable structure described above with one class of every finite size $k$. 
Let $A_K$ be the dense subset of $\omega$ which will have character $\{k,1): k \in K\}$ under $E$.  Then take $\omega \setminus A_K$ and partition it into exactly one class of size $k$ for $k \in K$.  This defines a faithfully coarsely computable structure $(\omega,R)$ with the desired character so that $R$ agrees with $E$ on the dense set $A_K$.
\end{proof}

\begin{lem} There is dense set $K$ with no infinite $\Sigma^0_2$ subset. 
\end{lem}

\begin{proof} This is just a generalization of the existence of an immune set. 
Let $S_1, S_2, \dots$ enumerate the $\Sigma^0_2$ sets and define $K$ to omit the least member of $S_i$ which is greater than $2^i$.  Then $K$ must contain at least $2^i - i$ of the first $2^i$ numbers and hence has density one. 
\end{proof}

It follows that not every faithful coarsely computable equivalence structure has a faithful generically computable copy. 

Finally, we show that there are equivalence structures which  do not have faithful coarsely computable copies.

\begin{thm} There is an infinite set $K \subset \omega$ such that if $\C = (\omega,R)$ is a computably enumerable equivalence structure such that 
  $\{x: |[x]_R\ = k\}$ has asymptotic density zero for any $k$, and such that if $D$ is a set of asymptotic density one, then $D$ is not a subset of 
$\{x: |[x]_R| \in K \}$. Thus any equivalence structure $\sA$ with character $\chi(\sA) \subset K \times \{1\}$ cannot be faithfully coarsely computable. 
\end{thm}

\begin{proof} Let $\C_e: = (\omega, S_e)$ be the $e^{th}$ computably enumerable equivalence structure.  That is, let $W_e$ be the $e^{th}$ computably enumerable set, and let $S_e$ be the reflexive, symmetric, transitive closure of $\{(x,y): \langle x,y \rangle \in W_e$. Let $[x]_e$ denote the equivalence class of $x$ in $\C_e$.  We need to meet the following requirement.

\medskip

{\bf Requirement $R_e$}: If  $\{x: |[x]_e|\ = k\}$ has asymptotic density zero for all $k$, then $\{x: |[x]_e| \in K\}$ does not have asymptotic density one. 

\medskip
 
We begin the construction with $K^0 = \omega$ and remove numbers at certain stages to accomplish the requirements.  At the same time, we need to ensure that $K$ is infinite.  So the construction will preserve an element of $K$ each time that it removes an infinite number of elements.  

We will show how to satisfy an individual requirement by the case $e=0$.  Let  $\sC = (\omega,S_0)$, let $S = S_0$, and consider the four sets $A_i = \{x: |[x]_S| = i\ mod\ 4\}$ for $i=0,1,2,3$. Since the union of the sets equals $\omega$, at least one of the sets, say $A_j$,  must have upper asymptotic density at least $1/4$. Let us suppose that $\{x: |[x]_R|\ = k\}$ has asymptotic density zero for all $k$, so that we need to take action on requirement $R_0$.  Then we will ensure that $K \cap \{i: i = j \ mod\ 4\} = \{4+j\}$; that is, we let $K^1 = \{4+j\} \cup \{k: k \neq j\ mod\ 4\}$ and maintain $K \cap \{i: i = j \ mod\ 4\} = \{4+j\}$ throughout the construction. 
Then $\{x: |[x]_R| \in K\}$ must have density at most $3/4$, so that it cannot contain any set $D$ has asymptotic density one.

The general construction of $K$ is in stages. After stage $e$, we will have designated, for certain $i \leq e$, a value $j(i)$ and corresponding set $A_i = \{x: |[x]_i| = j(i)\ mod\ 2^{i+2}\}$, so that for $i \neq h$, we have $A_i \cap A_h = \emptyset$. We will have removed $K_i = \{m: m = j(i)\ mod\ 2^{i+2}\}$ from $K$, except for $2^{i+2} + j(i)$, for such $i$, resulting in the set $K^s$. Note that we will have removed at most one set $K_i$  mod $2^{i+2}$ for each $i \leq e$, for a total of at most $2^e + 2^{e-1}  + \dotsb +  1 < 2^{e+1}$ classes mod $2^{e+2}$, resulting in the set $K^e$. Thus, there remain $2^{e+1}$ classes mod $2^{e+2}$ to work with, each disjoint from the previous classes.  At stage $e+1$, we will ensure Requirement $R_e$ (if necessary) by removing a set of class sizes from $K$. If there exists $k$ such that $\{x: |[x]_{e+1}| = k\}$ has positive measure, then we take no action.  If not, then we select $j = j(e+1) < 2^{e+3}$ such that $A_{e+1} = \{x: |[x]|_{e+1} = j\ mod\ 2^{e+3}\}$ has upper density at least $2^{-e-3}$ and we let $K_{e+1} = \{m: m = j(e+1)\ mod\ 2^{e+3}\}$.  If $K_{e+1}$ meets one of the previous classes $K_i$, then in fact $K_{e+1} \subset K_i$, so that we have already removed all but one element of $K_{e+1}$ from $K$ by stage $s$.  Otherwise, we remove 
$K_{e+1} = \{m: m = j\ mod\ 2^{e+3}\}$ from $K^e$, except for $2^{e+3} + j$, to obtain $K^{e+1}$. 

Let $K = \cap_s K^s$.  It remains to check that $K$ satisfies each Requirement $R_e$ and is an infinite set. 

First we show that action is taken infinitely often. Suppose, by way of contradiction, that no action is taken after stage $e$. Then $K$ will consist of a finite number of equivalence classes modulo $2^{e+2}$ plus a finite set. Thus $K$ will  be computable.  Hence there is some $i$ such that $\C_i$ consists of exactly one class of size $k$ for each $k \in K$. Thus at stage $i$, when we select $j$ such that 
$\{x: |[x]_i|= j \ mod\ 2^{i+2}\}$ has positive upper density in $\C_i$, and consider $K_i = \{m: m = j\ mod\ 2^{i+2}\}$, we would have $K_i \subset K \subset K^{i+1}$. But then we would have taken action and removed all but one value of $K_i$ from $K$.    

Next we need to check that $K$ is infinite.  Since action was taken infinitely often, we have preserved in $K$ an element $2^{i+2} + j(i)$ of $K_i$ for infinitely many $i$. Since the sets $\{K_i: i \in \omega\}$ are disjoint, this element is never removed at any later stage. Hence $K$ is infinite.

Now suppose that $\left\{x: \left|[x]_e\right|\ = k\right\}$ has asymptotic density zero for all $k$, and suppose, by way of contradiction, that $\{x: |[x]_e| \in K\}$ has 
asyptotic density one.  Then at stage $e$ of the construction we will have selected $j < 2^{e+2}$ such that $A_j = \{x: |[x]|_e = j\ mod\ 2^{e+2}\}$ has upper density at least $2^{-e-2}$, and defined \[K_e = \{m: m = j\ mod\ 2^{e+2}\}\].  Since $K \subset K^{e-1}$, it follows that $K_e$ is disjoint from all previous $K_i$. So we will remove all but one element of $K_e$ from $K$ at stage $e$.  It follows that $\{x: |[x]_e| \in K\}$ has lower density at most $1 - 2^{-e-2}$. 

Finally, suppose that $\sA = (\omega,S)$ has character $\chi(\sA) \subseteq K \times \{1\}$ and is faithfully coarsely computable. Let $\sC = (\omega,R)$ be a computable equivalence structure, and let $D$ be an $S$-faithful,  $R$-faithful set of density one such that $R$ and $S$ agree on $D$. Since $D$ is $S$-faithful, $D \subseteq \{x: |[x]|_S \in K\}$.   Since $R$ and $S$ agree on $D$, and $D$ is $R$-faithful, it follows that $D \subseteq \{x: |[x]|_R \in K\}$. Suppose first that there is some $k$ such that $\{x: |[x]_R| = k\}$ has positive lower density for some $k$.  Since $D$ contains at most one class of size $k$, this means that $D$ cannot have density one. Otherwise, by the first part of our theorem, $D \subseteq \{x: |[x]_R| \in K\}$ implies that $D$ cannot have asymptotic density one. 
\end{proof}

\section{Generically and Coarsely Computable Isomorphisms}

In this section, we consider isomorphisms that are \emph{generically} or \emph{coarsely} computable. 
So we first need to extend these notions from sets and relations to functions. 

\begin{dfn} Let $f: \omega \to \omega$ be a total function. 
\begin{enumerate}
\item We say that $f$ is \emph{generically computable} if there  is a partial computable function $\phi$ such that
  $\phi = f$ on the domain of $\phi$, and such that the domain of
  $\phi$ has asymptotic density 1.
\item We say that $f$ is \emph{coarsely computable} if there is a total computable function  $\phi$ such that $\{n: f(n) = \phi(n)\}$ has asymptotic density 1.
\end{enumerate}
\end{dfn}

It is easy to see that a set is generically computable if and only if $\chi_S$ is generically computable and likewise for coarsely computabile.

\begin{dfn}  Two structures $\sA$ and $\sB$ are said to be \emph{generically computably isomorphic} if there is an isomorphism $f: \sA \to \sB$ 
and a partial computable function $\theta$ such that both the domain and the range of $\theta$ have asymptotic density one, and $\theta(x) = f(x)$ whenever $\theta(x)$ is defined. 
\end{dfn}

\begin{prop} Two structures $\sA$ and $\sB$ are generically
 computably isomorphic if and only if there is an isomorphism $f: \sA \to \sB$ 
such that both $f$ and $f^{-1}$ are generically computable.
\end{prop}

\begin{proof} Suppose first that $\sA$ and $\sB$ are generically computably isomorphic and let $f$ and $\phi$ be given as in the defintion.  Then $f$ is certainly generically computable. Define the partial computable function $\psi$ to be $\phi^{-1}$, that is, $\psi(b) = a$ if $\phi(a) = b$.
Then if $\psi(b) = a$, it follows that $f(a) = b$ and therefore $f^{-1}(b) = a$. The domain of $\psi$ equals the range of $\phi$ and is therefore asymptotically dense. 

For the other direction, suppose that both $f$ and $f^{-1}$ are generically computable  isomorphisms.  Let $\theta$ and $\psi$ be partial computable functions with domains of asymptotic density one, such that
$\theta(a) = f(a)$ whenever it is defined, and $\psi(b ) = f^{-1}(b)$ whenever it is defined.
Then we may define an extension $\phi$ of $\theta$ by letting $\phi(a)$ equal either $\theta(a)$ or the
(unique) $b$ such that $\psi(b) = a$, whichever converges first. If both of these converge, then the value $b$
must equal $f(a)$. 
\end{proof}

\begin{dfn} \rule{0in}{0in}
  \begin{enumerate}
\item Let $f: \sA \to \sB$ be an isomorphism between two structures.  $f$ is said to be a \emph{coarsely computable} isomorphism if there is a  total computable function $\theta$ such the set $C = \{x: \theta(x) = f(x)\}$ is asymptotically dense and the image $f[C]$ also has asymptotic density one.  
\item $\sA$ and $\sB$ are said to be \emph{weakly coarsely computably isomorphic} if there is a set $C$ of asymptotic density one, a set isomorphism $f: \sA \to \sB$ and a total computable function $\theta$ which satisfy the following:
\begin{itemize}
\item[(i)] $C$ is the universe of a substructure $\sC$ of $\sA$;
\item[(ii)] $f(x) = \theta(x)$ for all $x \in C$;
\item[(iii)] $f[C]$ has asymptotic density one;
\item[(iv)] $\theta$ is an isomorphism from $\sC$ to its image.   
\end{itemize}
\end{enumerate}
\end{dfn}

For example, if $\sA$ and $\sB$ are equivalence structures, each having infinitely many classes of size 3, and the rest of $\sA$ and of $\sB$  consist of one class of size $4+n$ for each $n$, and it happens that the classes of size 3 in $\sA$ make up a dense computable set $C$ and the classes of size 3 in $\sB$ make up a dense computable set $D$, then we can define the computable map $\theta$ to map $\sC$ to $\sD$ preserving the classes, and to arbitrarily map the complements.  The set isomorphism $f$ can then agree with $\theta$ on $C$, but define an isomorphism of the complements, preserving the classes. In general, there may be no such $f$ which is computable.  

\begin{dfn} We say that $\sA = (\omega,E)$ has \emph{generic character} $K$ for a finite subset $K$ of $\omega \setminus \{0\}$ if, for each $k \in K$, the set $\sA(k)$ of elements of type $k$ has positive asymptotic density and the union $\bigcup_{k \in K} \sA(k)$ has asymptotic density 1. 
\end{dfn}

Thus if the generic character of $\sE$ is $\{k\}$ for some $k \leq \omega$, then the elements of $\sE$ of type $k$ has asymptotic density one. 

The classic example of a simple computable equivalence structure which is not computably categorical is one which consists of infinitely many classes of size one and infinitely many classes of size two.  Indeed, there are computable structures of this kind which are not computably isomorphic. We will call such an equivalence structure a \emph{$(1,2)$-structure}. The next result shows that under certain density conditions two such structures will be generically computably isomorphic. 

\begin{thm} Suppose that $\sA$ and $\sB$ are computable $(1,2)$ equivalence structures, each having generic character $\{2\}$.  Then $\sA$ is generically computably isomorphic to a computable structure in which the set of elements of size $2$ is computable, and therefore $\sA$ and $\sB$ are generically computably isomorphic. 
\end{thm}

\begin{proof} The elements in $\sA$ of type 2 form a computably enumerable set, so the classes of size 2 may be computably enumerated as $\{a_0,b_0\}, \{a_1,b_1\}, \dots$. That is, there is a computable enumeration of the set of pairs $\{\la x,y\ra: x \neq y\ \&\ E(x,y)\}$. At the same time our standard model $\sC$ can have the classes of size 2  make up a computable set of asymptotic density one, for example, the classes $\{n^2+i,n^2 +i+1\}$ where $1 \leq i < 2n$ for each $n \geq 1$; enumerate these in order as $\{c_0,d_0\}, \{c_1,d_1\}$, and so on.  Then the partial computable function $\phi$ may be defined so that $\phi(a_n) = c_n$ and $\phi(b_n) = d_n$; the inverse of $\phi$ is also partial computable.  This partial isomorphism  can be extended arbitrarily on the classes of size one to produce a generically computable isomorphism $f: \sA \to \sC$.

 For the next part, we will have as above a generically computable isomorphism $G: \sC \to \sB$ and a corresponding partial computable $\Psi$ mapping the asymptotically dense set of elements of type two from $\sC$ onto the elements of $\sB$  of type two.  Then the composition $\Psi \circ \Phi$ will be a partial computable function mapping the elements of $\sA$ of type two onto the elements of $\sB$ of type two and hence $G \circ F$ will be a generically computable isomorphism from $\sA$ to $\sB$.  
\end{proof}

This result can be generalized to structures having generic character $\{k\}$ and only finitely many classes of size $>k$.  On the other hand, if $\sA$ and $\sB$ have generic character $\{k\}$ but have infinitely many classes of sizes larger than $k$, then no similar result holds.

\begin{thm} \label{thm12} For any finite $k$, there exist computable $(1,2)$ structures $\sA$ and $\sC$, both having generic character $\{1\}$, which are not generically computably isomorphic. 
\end{thm}

\begin{proof} First we appeal to Proposition 2.15 of \cite{JS12} to get a simple computably enumerable set $B$ of asymptotic density zero. (Recall that $B$ is simple if and only if there is no infinite computably enumerable subset of $\omega \setminus B$.) Then, by Theorem 4.1 of \cite{CHR11}, there is a computable equivalence structure $\sA$ consisting of infinitely many classes of size two which make up the set $B$, together with infinitely many classes of size one.

  We compare this with some standard computable structure $\sC$ isomorphic to $\sA$ in which the classes of size two make up a computable set $D$ of asymptotic density zero, for example, the classes of size two could be of the form $\{n^2,n^2+1\}$ for $n\geq 1$. Now suppose, by way of contradiction, that there were a generically computable map $f: \sC \to \sA$ and a corresponding partial computable function $\phi$ such that the domain of $\phi$ has density one and $f(x) = \phi(x)$ whenever $\phi(x)$ is defined.  

Then the set $(\omega \setminus D) \cap Dom(\phi)$ must have asymptotic density one as the intersection of sets of density one, and it is also computably enumerable, since $D$ is a computable set. 
But then its image under $\phi$ is an infinite computably enumerable subset of $\omega \setminus B$, 
violating the assumption that $B$ is a simple set. 
\end{proof}

We observe that this result will also hold for $(1,k)$ structures, that is, equivalence structures consisting of infinitely many classes of size 1 and infinitely many classes of size $k>1$ for some finite $k$, since Theorem 4.1 of \cite{CHR11} also holds for $(1,k)$ structures. 

The notion of coarsely computable isomorphism is a weaker notion, as seen by the following.

\begin{thm} Let $\sA$ and $\sB$ be isomorphic equivalence structures with generic character $\{1\}$ (that is, the set of elements of $\sA$, and of $\sB$, of type one, both have asymptotic density one). Then $\sA$ and $\sB$ are coarsely computably isomorphic.
\end{thm}

\begin{proof} For any element $x$, let $[x]_A$ be the equivalence class of $x$ in $\sA$ and $[x]_B$ be the equivalence class of $x$ in $\sB$.
Let $U_A = \{x: |[x]_A| = 1\}$, let  $U_B = \{x: |[x]_B| = 1\}$, and let $U = U_A \cap U_B$. By assumption, $U_A$ and $U_B$ have asymptotic density one, and thus $U$ also has asymptotic density one.  Now the identity function $\Phi(x) = x$ is a total computable function and acts as an isomorphism of $U$ to itself.  We want to arbitrarily extend $\phi$ to an isomorphism $f: \sA \to \sB$.  
The only difficulty might be that $U_A \setminus U$ and $U_B \setminus U$
have different cardinalities, say, without loss of generality, that $U_B \setminus U$ is smaller.  Then we can remove from $U$ a subset of $U_B$ of density zero to produce a set $V \subset U$ of density one such that $U_A \setminus V$ and $U_B \setminus V$ have the same cardinality. 
 This will make $\sA \setminus V$ isomorphic to $\sB \setminus V$ so that we may extend $\phi$ from $V$ to an isomorphism $f$ from $\sA$ to $\sB$ which agrees with $\phi$ on the set $V$ of density one. 
\end{proof}

It is not clear whether this result can be extended, even to structures with generic character $\{2\}$.

Without the additional conditions on the density of substructures, computable equivalence structures which are not 
computably isomorphic are, in general, not coarsely computably isomorphic either.

Recall from \cite{CCHM06} that a computable equivalence structure $\sA$ is computably categorical if and only if one of the following holds: 

\begin{enumerate}
\item $\sA$ has only finitely many finite equivalence classes, or
\item $\sA$ has finitely many infinite classes, there is a bound on the size of the finite equivalence classes, and there is at most one $k$ such that $\sA$ has infinitely many classes of size $k$. 
\end{enumerate} 

xxxx

\begin{conj} If $\sA$ is a computable equivalence structure which is not computably categorical, then there exist computable copies $\sB$ and $\sC$ of $\sA$ which are not coarsely computably isomorphic. 
\end{conj}

Next we look at structures where the densities are positive but not 1. We will again focus here on $(1,2)$-structures.  From the examples seen so far, we might suspect that suspect that different densities pose a barrier to asymptotically computable isomorphism in such structures.  We will see that, at least for weakly coarsely computable isomorphism, it does not.

For any equivalence structure $\sA$, and any $n \leq \omega$, let $\sA(n) = \{x: |[x]| = n\}$.  The following lemma will be useful.

\begin{lem} \label{lem3.2} For any $\Delta^0_2$ real $q \in [0,1]$, there is a computable $(1,2)$-equivalence structure $\sA$ such that the asymptotic density of the elements of type one equals $q$.
\end{lem}

\begin{proof} Let $q = \lim_{n \to \infty} q_n$ where each $q_n$ is a dyadic rational and $q_n$ is not $0$ or $1$ for any $n$.  
We will define a computable increasing sequence $s_n$ and define the computable equivalence relation $\sA = (\omega,E)$ in stages $s$ on
 all numbers up to $s_n$ such that the relative number of elements of classes of size two is $q_n$.  For $n=0$, let $q_0 = i/j$ and let $s_0 = 2j$. 
Define $E$ up to $2j$ to have classes $\{0\}, \{1\},\dots,\{2i-1\}$ of size one and $\{2i,2i+1\}, \dots, \{2j-2,2i-1\}$ of size two. 
Given the definition of $E$ on $\{0,1,\dots,s_n-1\}$ such that there are $q_n s_n$ classes of size one, so that $(1-q_n) s_n$ is even and we may also assume that $s_n$ is even,  and given $s_{n+1} = i/j$, do the following. Let $s_{n+1} = j s_n$ and add $(i-q_n) s_n$ new classes of size one and $(j-i-1+q_n)s_n$ new classes of size two out of the numbers between $s_n$ and $j s_n$. Thus we end up with $i s_n$ out of $j s_n$ classes of size one, as desired. 
We  just observe that $i-q_n > 0$ since we assume that $i \geq 1$(since $q_{n+1} \neq 0$) and $q_n < 1$ and $j-i-1+q_n >0$ since  $j>i$ (because $q_{n+1} \neq 1$) and $q_n > 0$. 
\end{proof}

\begin{lem} \label{lem3.3} If two isomorphic computable equivalence structures $\sA$ and $\sB$ have bounded character, and for each $n \leq \omega$, $\sA(n)$ and $\sB(n)$ are computable, 
then $\sA$ and $\sB$ are computably isomorphic.  
\end{lem}

\begin{proof} We simply partition each structure into classes of a particular size $n$, and then observe that $\sA(n)$ is computably isomorphic to $\sB(n)$ for each $n$.
\end{proof}

\begin{thm} \label{thm3.4} Suppose that $\sA = (\omega,R)$ is a computable $(1,2)$-structure such that the asymptotic density of the elements of type one is a real $q$, so that the asymptotic density of the elements of type two is $1-q$, with $0 < q < 1$. Then $\sA$ is weakly coarsely computably isomorphic to some computable structure $\sC$ in which the set of elements of size $2$ is computable and has density $q$. 
\end{thm}

\begin{proof} We first build a computable equivalence structure $\sB$ isomorphic to and weakly coarsely computably isomorphic to $\sA$ where the density of $\mathcal{B}(1)$ is the same as that of $\sA(1)$, and where $\mathcal{B}(1)$ and $\mathcal{B}(2)$ are computable.  We have also a standard computable structure $\sC \cong \sA$ with $\sC(1)$ and $\sC(2)$ computable, and where $\sC(2)$ has density $q$.  By Lemma \ref{lem3.3}, we have $\sB$ computably isomorphic to $\sC$, so that $\sA$ is weakly coarsely computably isomorphic to $\sC$.

  To contstruct $\sB$, we let $A^s(2) = \left\{x \leq s: \exists y \leq s \left[ (y \neq x) \wedge (xRy)\right] \right\}$ and let $A^s(1) = s \setminus A^s(2)$.  Then for each $s$, $A^s(2) \subsetneq A(2)$ whereas $A(1) \cap s \subseteq A^s(1)$.  The idea of the proof is that classes  of size two are observable and that the sets $A^s(1)$ approximate $A(1)$.  Thus we will define $\sB = (\omega,R_B)$ so that $R_B$ is a subset of $R$ and differs from $R$ on a set of asymptotic density zero, so that we can use the identity as our set isomorphism.

  We define computable increasing sequences $(n_i)_{i<\omega}$ and $(s_i)_{i<\omega}$ with $2^i \leq n_i \leq s_i$ and define the relation $R_B$ for all pairs $(x,y)$ for all $x,y < n_i$ at stage $s_i$, so that $R_B$ is computable.  We will let $q_i = |A(1) \cap n_i|/n_i$, so that $\lim\limits_i q_i = q$.  Let $n_0 = 1 = s_0$.  Given $n_i$ and $s_i$, and having defined $R_B$ on all elements less than $n_i$ as well as some other elements less than $s_i$, and having defined $B(1)$ up to $n_i$, let $(n_{i+1},s_{i+1})$ be the least pair such that $|A^s(1) \cap n|/n < q+2^{-i}$. Now extend the definition of $R_B$ and of $B(1)$ as follows.  For any $x,y$ with $n_i < x < y < s$, let $xR_Bx$ if and only if $xRy$.  For $x$ such that $n_i \leq x < n_{i+1}$, put $x \in B(1)$ if there is no $y$ with $x<y<s_{i+1}$ such that $xRy$. For $y$ with $s_i \leq y < s_{i+1}$, put $y \in B(1)$ if there is $x \in B(1)$ such that $xRy$.  This is necessary to ensure that $B(1)$ is computable, so that we cannot change our mind about $[x]_B$ being a singleton once we have decided that it is. This also means that $B(1)$ will contain pairs $x,y$ of elements where $xRy$ but $y$ is much larger than $x$.

  It is clear that $A(1) \subset B(1)$ and it remains to calculate the density of $B(1) \setminus A(1)$.  Let $e_i = |A^{s_i}(1) \cap n_i \setminus A(1)|/n_i$; these are the only elements which may be put into $B(1)$ since they will have a partner larger than $n_i$.  Since  $|A^{s_i}(1) \cap n_i|/n_i < q+2^{-i}$, it follows that $e_i < q-q_i + 2^{-i}$.  Since $A(1)$ has asymptotic density $q$ and $n_i \geq 2^i$, it follows that $\lim\limits_i q_i = q$, and hence the set of elements where $R_B$ differs from $R$ has asymptotic density zero.

  Thus, the identity is a set-isomorphism which is an isomorphism between $\sA$ and $\sB$ on a set of asympotic density one, as desired. Note that, since $0 < q < 1$, and $B(1) \setminus A(1)$ has density zero, the set $B(1)$ will still have asymptotic density $q$. 
\end{proof}

We observe that this result will also hold for $(1,k)$ structures.  that is, equivalence structures consisting of infinitely many classes of size 1 and infinitely many classes of size $k>1$ for some finite $k$. 

\begin{lem} \label{lem3.4} Suppose that $A = \{a_0 < a_1 < \dotsb\}$ has positive asymptotic density $\alpha$ and that $\lim\limits_n |C \cap a_n|/n = 0$. Then $C$ has asymptotic density zero.
\end{lem}

\begin{proof} Since $A$ has positive density $\alpha$ and for  $A \cap a_n = \{a_0,\dots,a_{n-1}\}$, it follows that $|A \cap a_n|/a_n = n/a_n$ and thus $lim_n n/a_n = \alpha$.  Then 
\[
lim_n a_n/a_{n+1} = lim_n \frac {n+1}{a_{n+1}} / lim_n n{a_n} = \alpha/\alpha = 1.
\]

For any $i > a_0$, we have $a_n < i \leq a_{n+1}$ for some $n$.  Then $|C \cap a_n| \leq|C \cap i| \leq |C \cap a_{n+1}|$, so 

\[|(C \cap i|/i \leq |C \cap a_{n+1}|/a_n = |C \cap a_{n+1}|/a_{n+1} \cdot a_{n+1}/a_n,
\]
so that $lim_i |C \cap i|/i = 0$, as desired.
\end{proof}

\begin{lem} \label{lem3.5} Let $A$ and $B$ be subsets of $\omega$ having positive asymptotic densities $\alpha$ and $\beta$.  Suppose that $C \subset A$ and $D \subset C$ are computably enumerable sets, both of asymptotic density zero.  Then there is a computable isomorpihsm $f: \sA \to \sB$ such that $f[C]$ and $f^{-1}[D]$ each have asymptotic density zero. 
\end{lem}

\begin{proof} Let $A = \{a_0 < a_1 < \dotsb\}$ and $B = \{b_0 < b_1 < \dotsb\}$. Let $\{c_0,c_1,\dots\}$ be a computable enumeration of $C$ and let $\{d_0,d_1,\dots\}$ be a computable enumeration of $D$, both without repetition.  The goal is to define the map $f$ so that it maps $C$ to $D$ modulo asymptotic density zero. The function $f$ is defined in alternating stages as follows.  Map  $c_0$ to $d_0$.  If $a_0 = c_0$, then, of course, $f(a_0) = d_0$. So suppose $a_0 \neq c_0$.  If $b_0 \neq d_0$, then let $f(a_0) = b_0$ and otherwise let $f(a_0) = b_1$.  

Then at stage $s+1$, we define $f(a_{s+1})$ and $f(c_{s+1})$ as follows. If $f(c_{s+1})$ is not already defined, let $f(c_{s+1}) = d_j$ for the least $j$ such that $d_j$ is still available, that is,  we have not already defined $f(a) = d_j$ for some $a$.  Since we have only defined $s+1$ values of $f$, it follows that $j \leq s+1$.  If $f(a_{s+1})$ is not already defined, let $f(a_{s+1}) = b_i$ for the least $i$ such that $b_i$ is still available and note here that $i \leq s+1$. 

Since $D$ has density zero, it suffices to show that $f[C] \setminus D$ has asymptotic density zero. By Lemma \ref{lem3.4}, it is enough to show that $lim_n |(f[C] \setminus D) \cap b_n|/b_n = 0$. 

It follows from the construction that  
\[
(f[C] \setminus D) \cap b_n \subseteq \{f(a_i): i<n, a_i \in C \}.
\]
It now follows that 
\[ 
|f[C] \setminus D) \cap b_n| \leq |C \cap a_n|,
\]
 and therefore
\[
|f[C] \setminus D) \cap b_n|/b_n  \leq |C \cap a_n|/a_n \cdot a_n/b_n.
\]
Now we saw in the proof of Lemma \ref{lem3.4} that $lim_n n/a_n = \alpha$ if $\{a_0 < a_1 < \dotsb\}$ has asymptotic density $\alpha$, and similarly $lim_n n/b_n = \beta$, so that $lim_n a_n/b_n = \beta/\alpha$.  Since $lim_n |C \cap a_n|/a_n = 0$ and $lim_n a_n/b_n = \beta/\alpha$ exists, it follows that $lim_n |f[C] \setminus D) \cap b_n|/b_n  =  0$, as desired. 

For the other part we have $(f^{-1}[D] \setminus C) \cap a_n \subseteq  \{a_i: i < n, a_i \in C\}$, 

It now follows that 
\[
|f^{-1}[C] \setminus D) \cap a_n| \leq |C \cap a_n|, 
\]
and therefore
\[
|f[C] \setminus D) \cap a_n|/a_n  \leq  \cdot |C \cap a_n|/a_n \cdot a_n/n.
\]
Since $\lim\limits_n |C \cap a_n|/a_n = 0$ and $\lim\limits_n a_n/n = 1/\alpha$ exists, it follows that \[\lim\limits_n |f[C] \setminus D) \cap a_n|/a_n  =  0,\] as desired. 
\end{proof}

\begin{thm} Suppose that $\sA$ and $\sB$ are computable $(1,2)$ equivalence structures with universe $\omega$ such that the asymptotic density of $A(1)$ and $B(1)$ both equal the same computable real $q$.
Then $\sA$ and $\sB$ are weakly coarsely computably isomorphic. 
\end{thm}

\begin{proof} Let $\sA$, $\sB$ and $q$ be given as above. It follows from the proof of Theorem \ref{thm3.4} that there are computable structures $\sC$ and $\sD$ with universe $\omega$ such that the identity map is a weakly coarsely computable isomorphism between $\sA$ and $\sC$ and also between $\sB$ and $\sD$, with the additional property that $\sC(2) \subseteq \sA(2)$, $\sD(2) \subseteq \sB(2)$,  and both $\sC(1) \setminus \sA(1)$ and $\sD(1) \setminus \sB(1)$ have asymptotic density zero. Now, by Lemma \ref{lem3.5}, there is a computable isomorphism  $g_2: \sC(2) \to \sD(2)$, and a computable isomorphism $g_1$ from $\sC(1)$ to $\sD(1)$ such that  $g_1[\sC(1) \setminus \sA(1)]$ and $g_1^{-1}[\sD(1) \setminus \sB(1)]$ each have asymptotic density zero. Then the desired set isomorphism $g: \sA \to \sB$ is defined as follows.  Given $x \in \sA$, there are two cases.  If $x \in \sC(1)$, then $f(x) = g_1(x)$ and if $x \in \sC(2)$, then $f(x) = g_2(x)$.  Let $E = \sC(2) \cup (\sA(1) \cap G_1^{-1}[B(1)])$.  Then $\omega \setminus E = (\sC(1) \setminus \sA(1)) \cup (g_1^{-1}(D(1) \setminus B(1))$, and therefore has asymptotic density zero, so that $E$ has density one.  At the same time, $\omega \setminus f[E] = (\sD(1) \setminus \sB(1)) \cup g_1[\sC(1) \setminus \sA(1)])$, which has asymptotic density zero, so that $f[E]$ has asymptotic density one and thus $E$ has density one. Let $x,y \in E$.   It follows from the construction of Theorem \ref{thm3.4} that for any $x,y \in E$, $xR_Ay \iff xR_By$. It remains to check that $f$ is an isomorphism on the set $E$.  Let $x,y \in E$.  There are three cases, without loss of generality. 
First note that if $x \in (\sA(1) \cap g_1^{-1}(\sB(1))$, then $x \in \sC(1)$, so that $f(x) = g_1(x)$ and $g_1(x) \in \sB(1))$, and therefore $g(x) \in \sB(1)$.  

Case 1:  $x,y \in \sC(2)$. Then $f(x) = g_2(x)$ and $f(y) = g_2(y)$ and we have 
\[
xR_Ay \iff xR_Cx \iff g_2(x) R_D(x,y) \iff g_2(x) R_B(x,y),
\]
so that $xR_Ax \iff f(x) R_B f(y)$.

Case 2: $x \in \sA(1) \cap g_1^{-1}(\sB(1))$ and $y \in \sC(2)$. Then $y \in \sA(2)$, and therefore $\neg R_A(x,y)$. Now, by the remark above, $f(x) \in \sB(1)$, whereas $f(y) = g_2(y) \in \sD(2) \subseteq \sB(2)$ and therefore $f(y) \in \sB(2)$.  Hence we have   $\neg f(x) R_B f(y)$. 

Case 3:  
$x \neq y$ and both are in  $\sA(1) \cap g_1^{-1}[\sB(1)]$.  Then, since both are in $\sA(1)$, we have $\neg x R_A y$.  
By the remark above, $f(x),f(y) \in \sB(1)$ as well and therefore $\neg f(x) R_B f(y)$.

Thus $f$ acts as an isomorphism on the set $E$ of asymptotic density one.  This completes the proof that $\sA$ and $\sB$ are weakly coarsely computably isomorphic. 
\end{proof}

This result also extends to computable $(1,k)$-structures with all classes of size one or $k$, where $k$ is finite. 
We close with the following conjecture. 

\begin{conj}  Let $K = \{k_1,\dots,k_n\} \subseteq \omega \setminus \{0\}$ be a finite set and let $q_1,\dots,q_n$ be positive reals such that $q_1 + \dotsb + q_n = 1$. Let $\sA$ and $\sB$ be computable equivalence structures such that $\sA(k_i)$ and $\sB(k_i)$ have asymptotic density $q_i$ for each $i$.  Then $\sA$ and $\sB$ are weakly coarsely computably isomorphic. 
\end{conj}

\bibliographystyle{amsplain}
\bibliography{GenComp}
\end{document}